\documentclass{amsart}       

\usepackage{graphicx}
\usepackage{amsmath}
\usepackage{amssymb}
\usepackage[all]{xy}

\newtheorem{theorem}{Theorem}[section]

\newtheorem{proposition}[theorem]{Proposition}

\theoremstyle{definition}
\newtheorem{definition}[theorem]{Definition}

\newcommand{\B}{\ensuremath{\mathbf{B}}}
\newcommand{\A}{\ensuremath{\mathbf{A}}}
\newcommand{\Cat}{\ensuremath{\mathsf{Cat}}}
\newcommand{\Set}{\ensuremath{\mathsf{Set}}}

\newcommand{\Mon}{\ensuremath{\mathsf{Mon}}}
\newcommand{\C}{\ensuremath{\mathbf{C}}}

\newcommand{\PsCat}{\ensuremath{\mathbf{PsCat}}}
\newcommand{\Bla}{\ensuremath{\mathbf{Bla}}}

\DeclareMathOperator{\dom}{dom}
\DeclareMathOperator{\cod}{cod}
\DeclareMathOperator{\map}{map}
\DeclareMathOperator{\cell}{2-cellstruct}

\begin{document}

\title{On categories with arbitrary 2-cell structures}\thanks{The author acknowledge Fundação para a Ciência e a Tecnologia (FCT) for its financial support via the project CDRSP Base Funding (DOI: 10.54499/UIDB/04044/2020). This research work was also supported by the Portuguese Foundation for Science and Technology FCT/MCTES (PIDDAC) through the following Projects: Associate Laboratory ARISE LA/P/0112/2020; UIDP/04044/2020; UIDB/04044/2020; PAMI - ROTEIRO/0328/2013 (N° 022158); MATIS (CENTRO-01-0145-FEDER-000014 - 3362); Generative.Thermodynamic; FruitPV; by CDRSP and ESTG from the Polytechnic Institute of Leiria.}


\author{Nelson Martins-Ferreira}





\begin{abstract}
When a category is equipped with a 2-cell structure it becomes a sesquicategory but not necessarily a 2-category. It is widely accepted that the latter property is equivalent to the middle interchange law. However, little attention has been given to the study of the category of all 2-cell structures (seen as sesquicategories with a fixed underlying base category) other than as a generalization for 2-categories. 
The purpose of this work is to highlight the significance of such a study, which can prove valuable in identifying intrinsic features pertaining to the base category. These ideas are expanded upon through the guiding example of the category of monoids. Specifically, when a monoid is viewed as a one-object category, its 2-cell structures resemble semibimodules.

\keywords{Pseudomonoid \and monoid transformation \and 2-cell structure \and sesquicategory \and 2-category monoidal category \and bicategory \and double category}
\end{abstract}

\date{Wed, 12Jun2024}

\maketitle

\section{Introduction}

This work continues a project initiated by the author sixteen years ago with the pre-print \cite{cmucpreprint2008}, followed by a sequence of papers and pre-prints \cite{preprint2009,NMF.15,NMF.17}. It originated with the observation that a pseudocategory \cite{NMF.06}, instead of being defined as a categorical structure internal to a 2-category, could be extended to a categorical structure internal to a sesquicatery \cite{Stell1994}.
Moreover, it was noted that considering a category equipped with a 2-cell structure is more appropriate than just a sesquicategory. This distinction might seem artificial since every category equipped with a 2-cell structure is precisely a sesquicategory.
However, the main goal pursued in this project is obtained by considering a fixed category and study all different 2-cell structures over it. This study is important because if varying the 2-cell structure on a category $\C$ it makes the category of pseudocategories internal to $\C$ to vary as well. Each category is trivially equipped with two different 2-cell structures, namely discrete and co-discrete. If $\C$ is equipped with the discrete 2-cell structure then a pseudocategory internal to $\C$ is an internal category while if $\C$ is equipped with the co-discrete 2-cell structure then a pseudocategory internal to $\C$ is a pre-category \cite{janelidze2003}. Moreover, if $\C$ is equivalent to $\Cat(\B)$ for some category $\B$ then $\C$ is equipped with a natural 2-cell structure and pseudocategories internal to $\C$ are equivalent to pseudo-double categories internal to $\B$ (see e.g. \cite{NMF.04}).

Suppose we could parameterize a family of 2-cell structures over a category $\C=\Cat(\B)$ using the unit interval, with the discrete and co-discrete structures indexed by 0 and 1, respectively, and the natural 2-cell structure of internal natural transformations indexed by the number one-half. In this scenario, we could envision a continuous deformation starting at $t=0$ with a strict double-category as studied by Grandis and Paré \cite{grandis1996}, transitioning into a pseudo-double category (see e.g. \cite{grandis2004,fiore2007}) as $t$ evolves from $0$ to $\frac{1}{2}$. Subsequently, this transition could continue, eventually leading to an extremely lax version where the unitality and associativity axioms are no longer present, not even up to isomorphism, as $t$ evolves from $\frac{1}{2}$ to $1$.
The main purpose of this paper is to give evidence of such a possibility with the category of monoids as  guiding example.

Inspired by the foundational work of MacLane, Ehresmann, and Bénabou, the concept of internal pseudocategory, as introduced in \cite{NMF.06}, arises as a versatile framework within the realm of categorical structures.
This concept, developed within the context of a 2-category, serves as a unified structure encompassing monoidal categories, double categories, and bicategories \cite{benabou1970, ehresmann1963,  maclane1963}. 
In alignment with Grothendieck's approach, endorsed by Brown and Janelidze \cite{brown1976,grothendieck1984,janelidze1990}, that embraces a geometric perspective on algebraic categories, the study of internal pseudocategories has a potential to provide valuable insights into the intrinsic structure of the base category.
In particular, when interpreted within $\Cat$, the 2-category of all categories, functors, and natural transformations, pseudocategories yield the familiar structures of monoidal categories, double categories, and bicategories. 
For instance, an internal bicategory is the same as an internal pseudo-double category in which the object of vertical morphisms is terminal. An internal monoidal category is the same as a pseudomonoid \cite{street97,mccrudden2000}, which is the same as a pseudocategory with a terminal object of objects. Moreover, a double category internal to a category $\B$ corresponds to an internal category in $\C=\Cat(\B)$ and can be viewed as a strict pseudocategory within $\C$.

While exploring categories of the form $\Cat(\B)$ or more generally abstract 2-categories offers diverse examples and situations, an even richer diversity emerges when considering categories with a 2-cell structure that is not bound by the middle interchange law. This departure, as suggested in \cite{cmucpreprint2008}, is akin to transitioning from the study of abelian groups to the broader realm of groups. Jordan's exploration of non-commutative lattices \cite{jordan1949} highlights the significance of considering structures beyond commutativity, indicating that restricting the focus to the commutative case would significantly limit the depth and richness of the theory. Similarly, in the realm of 2-cell structures, relaxing the middle-interchange law holds promise for expanding horizons and uncovering new avenues of exploration.

An important challenge arises when considering arbitrary or unnatural 2-cell structures within the context of internal pseudocategories and pseudomonoids: the inherent lack of coherence. To address this challenge, it becomes necessary to introduce additional coherence diagrams. While Mac Lane's Theorem \cite{maclane1992} provides a solid foundation for coherence in general, it primarily focuses on natural 2-cell structures, leaving out considerations for arbitrary structures. However, we can extend the concept of coherence to encompass the commutation of specific diagrams related to naturality, thus broadening the theorem's scope. This extension hints at the existence of a finite set of coherence diagrams that could potentially extend the theorem's applicability to arbitrary 2-cell structures.

For example, the requirement for the commutativity of the pentagon diagram:
\begin{equation}
    \xymatrix@C=-1.2pc{& i\oplus(j\oplus (k\oplus l))\ar@{=>}[rr] \ar@{=>}[ld] && i\oplus((j\oplus k)\oplus l) \ar@{=>}[rd]\\
(i\oplus j)\oplus (k\oplus l) \ar@{=>}[rrd] &&&& (i\oplus (j\oplus k))\oplus l \ar@{=>}[lld]\\
&&  ((i\oplus j)\oplus k)\oplus l}
\end{equation}
together with commutativity of the triangles:

\begin{align}
    \xymatrix{0\oplus(i\oplus j)\ar@{=>}[rr]\ar@{=>}[rd] && (0\oplus i)\oplus j \ar@{=>}[ld]\\
& i\oplus j}\\
\xymatrix{i\oplus(0\oplus j)\ar@{=>}[rr]\ar@{=>}[rd] && (i\oplus 0)\oplus j \ar@{=>}[ld]\\
& i\oplus j}\\
\xymatrix{i\oplus(j\oplus 0)\ar@{=>}[rr]\ar@{=>}[rd] && (i\oplus j)\oplus 0 \ar@{=>}[ld]\\
& i\oplus j}
\end{align}
and the further pentagons (little pentagon and degenerated squares):
\begin{align}
    \xymatrix{& 0\oplus(i\oplus 0)\ar@{=>}[rr] \ar@{=>}[ld] && (0\oplus i)\oplus 0 \ar@{=>}[rd]\\
0\oplus i \ar@{=>}[rrd] &&&& i\oplus 0 \ar@{=>}[lld]\\
&&  i}
\end{align}
\begin{align}
\xymatrix{& i\oplus(0\oplus 0)\ar@{=>}[rr] \ar@{=}[ld] && (i\oplus 0)\oplus 0 \ar@{=>}[rd]\\
i\oplus 0 \ar@{=>}[rrd] &&&& i\oplus 0 \ar@{=>}[lld]\\
&&  i}
\end{align}
\begin{align}
\xymatrix{& 0\oplus(0\oplus i)\ar@{=>}[rr] \ar@{=>}[ld] && (0\oplus 0)\oplus i \ar@{=}[rd]\\
0\oplus i \ar@{=>}[rrd] &&&& i\oplus 0 \ar@{=>}[lld]\\
&&  i}
\end{align}
would be a possibility. However, as it is well-known if the 2-cells are natural then the first pentagon and second triangle are sufficient to ensure coherence (it is of course possible that the list above is already redundant, even for abstract 2-cells). 
Beyond addressing coherence, ensuring the naturality of 2-cells may require the introduction of additional diagrams. This suggests the potential necessity for a restatement or refinement of the coherence theorem, which could lead to a deeper comprehension of coherence within more general 2-cell structures. Exploring these avenues remains a subject for future investigation, deserving further scrutiny and attention (for example \cite{Brin05} points out a possible different direction).

This paper presents a refined characterization of $\cell(\C)$, the category of 2-cell structures over a fixed category $\C$, with a focus on extending the results of \cite{cmucpreprint2008}. The paper also introduces Theorem \ref{thm: arbitrary 2-cell structs}, a general procedure for constructing parameterized 2-cell structures over arbitrary base categories. Throughout, the concepts are illustrated using the category of monoids as a guiding example.

\section{A category with a 2-cell structure}

This section revisits the concept of a category equipped with a 2-cell structure, as introduced in \cite{cmucpreprint2008} (see also \cite{NMF.15}). It is important to note that although a category with a 2-cell structure is the same as a sesquicategory, emphasis is placed on the fact that each category $\C$, treated as a base category, gives rise to a new category, $\cell(\C)$, whose objects consist of 2-cell structures over $\C$ along with the natural transformations between them.
Furthermore, given a fixed base category $\C$ and a 2-cell structure $H=(H,\dom,\cod,0,+)$, see below, the category $\PsCat(\C,H)$ of pseudocategories internal to $\C$ and relative to the 2-cell structure $H$ offers valuable insights into the intrinsic properties of the base category. This framework enables the definition of an equivalence between 2-cell structures $H$ and $H'$ whenever the categories $\PsCat(\C,H)$ and $\PsCat(\C,H')$, both of which are tetra-categories \cite{NMF.10a}, are equivalent.

\begin{definition}
A category $\C$ is said to be equipped with a 2-cell-structure if  together with $\C$ and 
\begin{equation}
    \hom_{\C}\colon{\C^{\text{op}}\times \C\to \Set} 
\end{equation}
is given a bifunctor 
\begin{equation}
    H\colon{\C^{\text{op}}\times \C\to \Set} 
\end{equation}
and natural transformations
\begin{eqnarray}
    \dom,\cod \colon H\Longrightarrow  \hom_{\C}\\
    0 \colon \hom_{\C}\Longrightarrow  H\\
   + \colon H\times_{\langle \dom,\cod\rangle} H\Longrightarrow  H 
\end{eqnarray}
    making the diagram 
    \begin{equation}
\xymatrix{H\times_{\langle \dom,\cod\rangle} H \ar@{=>}[r]^(.7){+} & H \ar@{=>}@<1.5ex>[r]^(.4){\dom} \ar@{=>}@<-1.5ex>[r]_(.4){\cod} & \hom_{\C} \ar@{=>}@<0ex>[l]|(.6){0}}
    \end{equation}
    an internal category in $\Set^{\C^{\text{op}}\times \C}$, i.e. an object in $\Cat(\Set^{\C^{\text{op}}\times \C})$.
\end{definition}

For a fixed category $\C$, a morphism from a 2-cell structure $(H,\dom,\cod,0,+)$ to a 2-cell structure $(H',\dom,\cod,0,+)$ is a natural transformation $\tau\colon{H\Longrightarrow H'}$ inducing a morphism at the level of internal categories in $\Set^{\C^{\text{op}}\times \C}$ as illustrated.
    \begin{equation}
\xymatrix{H\times_{\langle \dom,\cod\rangle} H \ar@{=>}[r]^(.7){+} \ar@{=>}[d] & H \ar@{=>}@<1.5ex>[r]^(.4){\dom} \ar@{=>}@<-1.5ex>[r]_(.4){\cod} \ar@{=>}[d]_{\tau} & \hom_{\C} \ar@{=>}@<0ex>[l]|{0} \ar@{=}[d]\\
H'\times_{\langle \dom,\cod\rangle} H' \ar@{=>}[r]^(.7){+} & H' \ar@{=>}@<1.5ex>[r]^(.4){\dom} \ar@{=>}@<-1.5ex>[r]_(.4){\cod} & \hom_{\C} \ar@{=>}@<0ex>[l]|{0} 
}
    \end{equation}

Once again, the following definitions are extracted from \cite{cmucpreprint2008}. For the sake of simplicity, with an abuse of notation, expressions of the form $H(f,g)(x)$ will be written as $gxf\in H(A',B')$ for every $f\colon{A'\to A}$, $g\colon{B\to B'}$ and $x\in H(A,B)$, as illustrated
\begin{equation}
    \xymatrix{A \ar@{=>}[r]^{x} & B\ar[d]^{g}\\A'\ar[u]^{f}\ar@{=>}[r]_{gxf} & B'}
\end{equation}
see also \cite{NMF.15} for more details. In the same way as the elements in $\hom_{\C}(A,B)$ are called morphisms from $A$ to $B$, the elements $x\in H(A,B)$ will be called 2-cells from $A$ to $B$. Further notions pertained to natural and invertible 2-cells as well as commutator 2-cells are referred to subsection 5.3 in \cite{preprint2009} but will not be needed here.

\begin{definition}\label{def: natural wrt}
    Let $(H,\dom,\cod,0,+)$ be a 2-cell structure over a category $\C$. A 2-cell $x \in H(A,B)$ is said to be natural with respect to a 2-cell  $y\in H(X,A)$ if the equation $\cod(x)y+x\dom(y)=x\cod(y)+\dom(x)y$ holds good. Moreover, a 2-cell $x\in H(A,B)$ is said to be natural if it is natural with respect to all possible 2-cells $y\in H(X,A)$ for all possible objects $X$ in $\C$. 
\end{definition}

The equation $\cod(x)y+x\dom(y)=x\cod(y)+\dom(x)y$ is illustrated as 
\[\xymatrix{
{X}
\ar@<1ex>@/^1pc/[rr]^{\dom(y)}="1"
\ar@<-1ex>@/_1pc/[rr]_{\cod(y)}="2"   &&  {A} 
\ar@<1ex>@/^1pc/[rr]^{\dom(x)}="3"
\ar@<-1ex>@/_1pc/[rr]_{\cod(x)}="4"   &&  {B} \\
\ar@{=>}_{y } "1";"2"
\ar@{=>}_{x } "3";"4"
}
\]
and it is clear that the horizontal composition (or tensor composition, denoted as $x\otimes y$ when it exists) is defined only for those horizontally composable pairs $(x,y)$ in which $x$ is natural with respect to $y$. 
In such cases, the tensor (or horizontal composition) is defined as either $x\otimes y=\cod(x)y+x\dom(y)$ or $x\otimes y=x\cod(y)+\dom(x)y$, where $\dom$ and $\cod$ are as displayed.
\begin{equation*}
    \xymatrix{{X} 
\ar@<1ex>@/^1pc/[rr]^{\dom(x)\dom(y)}="2"
\ar@<-1ex>@/_1pc/[rr]_{\cod(x)\cod(y) }="4"     
&& {B} & \\
\ar@{=>}|{x\otimes y } "2";"4"}
\end{equation*}

A 2-cell structure $H=(H,\dom,\cod,0,+)$ is considered natural if every 2-cell within it satisfies the naturality condition, implying the validity of the middle interchange law. This property promotes $(\C,H)$ from a sesquicategory to a 2-category.

The rationale behind adopting additive notation is twofold. Firstly, it facilitates the distinction between vertical and horizontal composition. While horizontal composition, when defined, typically exhibits tensorial or multiplicative behavior, vertical composition is additive in nature, as illustrated in our example of monoids discussed below. Secondly, the functoriality of $H$ allows for the interpretation of morphisms in $\hom_{\C}(A,B)$ as scalars and 2-cells in $H(A,B)$ as vectors, with wiskering composition akin to scalar multiplication \cite{NMF.15}, see also Proposition \ref{prop: one object} below.

When $\C$ is the category of monoids and monoid homomorphisms we have, for every pair of monoids $(A,B)$, $\hom(A,B)$ as the set of monoid homomorphisms from $A$ to $B$. In the same spirit, denote by $\map(A,B)$ the monoid of maps (set theoretical functions) from the underlying set of $A$ to the underlying set of $B$. As usual, the zero map is denoted by $0_{A,B}$ and the monoid operation on $\map(A,B)$ is obtained by component-wise addition. In order to equip $\Mon$  with a 2-cell structure  a bifunctor $H$ must be specified. There are obviously many possibilities. The structure considered here is inspired by the  canonical structure which is obtained when considering a monoid as a one object category but it is not comparable to it (see list of examples in Section \ref{sec: examples}). For each pair of monoids $(A,B)$ let us denote by $H(A,B)$ the subset of $\map(A,B)\times \hom(A,B)$ consisting on those pairs $(t,f)$ with $t\colon{A\to B}$ a map and $f\colon{A\to B}$ a monoid homomorphism  such that the sum $t+f$ is a monoid homomorphism. In other words, $(t,f)\in H(A,B)$ if and only if $t(0)=0$ and for every $x,y\in A$ the condition
\begin{equation}\label{eq: t(x+y)}    t(x+y)+f(x)+f(y)=t(x)+f(x)+t(y)+f(y)
\end{equation}
is satisfied. Furthermore, $\dom(t,f)=f$, $\cod(t,f)=t+f$, $0_f=(0_{A,B},f)$, and
\begin{equation}
    (s,g)+(t,f)=(s+t,f)
\end{equation}
provided $g=t+f$. It is straightforward checking that $(H,\dom,\cod,0,+)$, as specified, is a 2-cell structure over $\Mon$. 

It is perhaps worthwhile noting the particular case of groups. Indeed, when $B$ is a group then $(t,f)\in H(A,B)$ if and only if $$t(x+y)=t(x)+f(x)+t(y)-f(x)$$ for every $x,y\in A$, which corresponds to an instance of the well-known notion of a crossed homomorphism as soon as the conjugation $f(x)+t(y)-f(x)$ is written as an action  $f(x)\cdot t(y)=f(x)+t(y)-f(x)$ so that $t(x+y)=t(x)+f(x)\cdot t(y)$.

In general, these 2-cells are not natural. In an abstract category with a 2-cell structure only a relative notion of naturality exists (Definition \ref{def: natural wrt}). Two horizontally composable 2-cells are natural with respect to each other precisely when their horizontal composition is defined via the middle interchange law.

The 2-cell structure in the category of monoids as described above is a situation where not every pair of 2-cells is horizontally composable even though their source and target would suggest a composition. In monoids, for $x=(t,f)$ and $y=(t',f')$ as displayed below, to say that the pair $(x,y)$ is horizontally composable is the same as saying that $x$ is natural with respect to $y$ in the sense of Definition \ref{def: natural wrt} which holds if and only if 
\begin{equation}\label{eq: t(t'+f')}
t(t'+f')+ft'=tt'+ft'+tf'.
\end{equation}
When that is the case we have $(t,f)\otimes(t',f')=(t\otimes t',ff')$ with $t\otimes t'=tt'+ft'+tf'$ or $t\otimes t'=t(t'+f')+ft'$ as illustrated
\begin{equation}\label{eq: tensor comp t't}
\xymatrix{
{A}
\ar@<1ex>@/^1pc/[rr]^{f'}="1"
\ar@<-1ex>@/_1pc/[rr]_{t'+f'}="2"   &&  {B} 
\ar@<1ex>@/^1pc/[rr]^{f}="3"
\ar@<-1ex>@/_1pc/[rr]_{t+f}="4"   &&  {C} \\
\ar@{=>}|{(t',f')} "1";"2"
\ar@{=>}|{(t,f) } "3";"4"
}
\mapsto\xymatrix{{A} 
\ar@<1ex>@/^1pc/[rr]^{f f'}="2"
\ar@<-1ex>@/_1pc/[rr]_{g g' }="4"     && {C} & \\
\ar@{=>}|{(t\otimes t',ff') } "2";"4"}
\end{equation}
with $g=t+f$ and $g'=t'+f'$. In virtue of equation (\ref{eq: t(x+y)}), if the monoid $C$ admits right cancellation, then any pair $(x,y)$ as described above is horizontally composable. However, this condition does not hold in general. Thus we observe that the 2-cell structure under consideration is not natural. Consequently, the category $\Mon$ with this particular 2-cell structure does not qualify as a 2-category. 
 This outcome is not unexpected, as we have significantly abstracted from the conventional perspective of treating a monoid as a one object category to endow $\Mon$ with a natural 2-cell structure (see list of examples in Section \ref{sec: examples}).
Nonetheless, as we will demonstrate, this approach yields a coherent and meaningful framework to operate within. Moreover, it offers the advantage of generating a richer concept of internal pseudomonoid compared to one restricted to natural 2-cell structures. This aspect is addressed in a forthcoming paper in which a characterization for internal pseudomonoids with respect to a parameterized 2-cell structure is given.

Here is a simple example to illustrate the fact that $\Mon$, equipped with the 2-cell structure detailed before, is not a 2-category. Take $A=B=C$ as the linear chain three element monoid, i.e., $(\{1,2,3\},\max,1)$, and represent maps $t\colon{\{1,2,3\}\to \{1,2,3\}}$ as vectors $t=(t(1),t(2),t(3))$. Then a quick computation shows that $t=t'=(1,3,2)$ and $f=f'=(1,2,3)$ yields, on the one hand  
\begin{equation}
    t(t'+f')+ft'=\max(t(\max(t',f')),ft')=(1,3,2)
\end{equation}
while on the other hand it yields
\begin{equation}
    tt'+ft'+tf'=\max(tt',\max(ft',tf'))=(1,3,3)
\end{equation}
showing that the middle interchange law does not hold in general. Of course, with $C$ a commutative monoid, every map $t$ with $t(x+y)=t(x)+t(y)$ satisfies conditions $(\ref{eq: t(x+y)})$  and $(\ref{eq: t(t'+f')})$ so that, by fixing $f=1_C$, the counter-example had to be found with a non-monotone map $t$ that would satisfy equation $(\ref{eq: t(x+y)})$ but not equation $(\ref{eq: t(t'+f')})$. 

\section{Characterizing the category of 2-cell structures over a fixed base category}

The following result is a slight improvement of the characterization detailed in the preprint \cite{cmucpreprint2008} which is at the origin of this project.
It underscores the rationale for adopting additive notation for the vertical composition of 2-cells. Additionally, it elucidates how wiskering composition reflects both left and right actions of morphisms on 2-cells, further solidifying the analogy between 2-cells and morphisms, reminiscent of the relationship between vectors and scalars in a vector space.

\begin{proposition}\label{prop: 2-cell structs characterization}
    Let $\C$ be a category and denote by $\hom$ its hom bifunctor. The category $\cell(\C)$ of all 2-cell structures over the base category $\C$ is isomorphic to the category whose objects are families of tuples $(H,\dom,\cod,0,+)_{A,B}$, indexed by pairs of objects $(A,B)$ in $\C$, where each $H_{A,B}$ is a set and is denoted as $H(A,B)$, while each $\dom_{A,B}$, $\cod_{A,B}$, $0_{A,B}$ ,$+_{A,B}$ is a map with domain and codomain as displayed
        \begin{equation}
\xymatrix{H(A,B)\times_{\hom(A,B)} H(A,B) \ar@{->}[r]^(.7){+_{A,B}} & H(A,B) \ar@{>}@<1.5ex>[r]^(.4){\dom_{A,B}} \ar@{->}@<-1.5ex>[r]_(.4){\cod_{A,B}} & \hom(A,B) \ar@{->}@<0ex>[l]|(.5){0_{A,B}}},
    \end{equation}
 together with an indexed family of maps 
    \begin{equation}\label{eq: mu}
        \mu_{A',A,B,B'}\colon{\hom(B,B')\times H(A,B) \times \hom(A',A)\to H(A',B')}
    \end{equation}
    satisfying the following conditions
    \begin{align}
        \dom(\mu(u,x,v))&=u\dom(x)v\\
        \cod(\mu(u,x,v))&=u\cod(x)v\\
        \mu(u,0_{A,B}(f),v)&=0_{A',B'}(ufv)\\
        \mu(u,x',v)+\mu(u,x,v)&=\mu(u,x'+x,v)\\
        \mu(u',\mu(u,x,v),v')&=\mu(u'u,x,vv')\\
        \mu(1_B,x,1_A)&=x\\
        \dom(0_{A,B}(f))=f&=\cod(0_{A,B}(f))\\
        \dom(x'+x)=\dom(x)&,\quad \cod(x'+x)=\cod(x')\\
        0_{A,B}(\cod(x))+x=x&=x+0_{A,B}(\dom(x))\\
        x''+(x'+x)&=(x''+x')+x
    \end{align}
    for all appropriate morphisms $u$, $v$, $u'$, $v'$, $f$ in $\C$ and $x'',x',x\in H(A,B)$ with $\dom(x'')=\cod(x')$ and $\dom(x')=\cod(x)$; the subscripts in $\dom_{A,B}$,  $\cod_{A,B}$, $+_{A,B}$ and $\mu_{A',A,B,B'}$ have been removed for readability.

    A morphism in $\cell(\C)$ is an indexed family of maps
    \begin{equation}
        \tau_{A,B}\colon{H(A,B)\to H'(A,B)}
    \end{equation}
    such that
    \begin{align}
        \tau_{A',B'}(\mu_{A',A,B,B'}(u,x,v))&=\mu'_{A',A,B,B'}(u,\tau_{A,B}(x),v)\\
        \dom'(\tau_{A,B}(x))&=\dom_{A,B}(x)\\
        \cod'(\tau_{A,B}(x))&=\cod_{A,B}(x)\\
        \tau_{A,B}(0_{A,B}(f))&=0'_{A,B}(f)\\
        \tau_{A,B}(x'+x)&=\tau_{A,B}(x')+'\tau_{A,B}(x)
    \end{align}
    for all objects $A',A,B,B'$ and morphisms $u\colon{B\to B'}$, $f\colon{A\to B}$, $v\colon{A'\to A}$ in $\C$ and for all $x\in H(A,B)$.
\end{proposition}
\begin{proof}
    See Proposition 5.3 in \cite{preprint2009}.
\end{proof}

The scenario outlined earlier in the context of monoids provides a clear example where all conditions can be readily verified. Specifically, for any two monoids $A$ and $B$, the set $H(A,B)$ comprises pairs $(t,f)$, where $f\colon A\to B$ is a monoid homomorphism and $t\colon A\to B$ is a set-theoretical map from the underlying set of $A$ to the underlying set of $B$, satisfying the condition that the map $t+f\colon A\to B$ is a monoid homomorphism. This condition permits defining $\cod(t,f)$ as $t+f$ and suggests setting $0_{A,B}(f)=(0_{A,B},f)$ and $(s,t+f)+(t,f)=(s+t,f)$, a formula reminiscent of those encountered in internal groups and crossed modules. For further illustrative examples in a similar vein, refer to \cite{NMF.15}.

The key observation in describing a 2-cell structure as suggested in Proposition~\ref{prop: 2-cell structs characterization}  is that a  bifunctor $H$, from $\C^{\text{op}}\times \C$ to $\Set$, is nothing but a family of sets $H(A,B)$ indexed by pairs of elements in $\C$ together with a family of maps $\mu_{A',A,B,B'}$ as displayed in equation $(\ref{eq: mu})$ such that $$\mu_{A',A,B,B'}(u,x,v)=H(v,u)(x)\in H(A',B')$$ which, abusively, is sometimes represented as $\mu(u,x,v)=uxv$ as illustrated.

\begin{equation}
    \xymatrix{A \ar@{=>}[r]^{x} & B\ar[d]^{u}\\A'\ar[u]^{v}\ar@{=>}[r]_{uxv} & B'}
\end{equation}

The particular situation of a category with a single object (seen as a monoid) is worthwhile mentioning. Let $\C=(M,\cdot,1)$ be a monoid considered as a one object category. The category of $M$-semibimodules is fully embedded in $\cell(\C)$.

\begin{proposition}\label{prop: one object}
    Let $\C=(M,\cdot,1)$ be a one object category seen as a monoid. The category whose objects are  pairs $(A,\mu)$ with $A=(A,+,0)$  a monoid and $\mu\colon{M\times A\times M} \to A$ a map such that for all $a,a'\in A$, $u,u',v,v'\in M$:
    \begin{align}
        \mu(u,0,v)&=0\label{eq: mu41}\\
        \mu(u,a',v)+\mu(u,a,v)&=\mu(u,a'+a,v)\\
        \mu(u',\mu(u,a,v),v')&=\mu(u'u,a,vv')\label{eq: 43}\\
        \mu(1,a,1)&=a\label{eq: mu44}
    \end{align}
    is fully embedded in $\cell(\C)$ by taking $H=M\times A\times M$, $\dom(g,a,f)=f$, $\cod(g,a,f)=g$, $0(f)=(f,0,f)$, $(h,a_1,g)+(g,a_2,f)=(h,a_1+a_2,f)$, for all $a_1,a_2\in A$ and $f,g,h\in M$. Moreover, $(g,a,f)\in H$ is natural with respect to $(g',a',f')\in H$ if and only if 
    \begin{equation}
        \mu(1,a,g')+\mu(f,a',1)=\mu(g,a',1)+\mu(1,a,f').
    \end{equation}
\end{proposition}
\begin{proof}
    Following Proposition \ref{prop: 2-cell structs characterization} it remains to specify the wiskering composition which is given as displayed.
    \begin{align*}
        M\times ( M\times A\times M )\times M\to M\times A\times M\\
        (u,(g,x,f),v)\mapsto (ugv,\mu(u,x,v),ufv)
    \end{align*}
    The remaining details are straightforward verification and are omitted.
\end{proof}

If by an $M$-semibimodule we mean a system $(A,0,+,\mu)$ with $(A,+,0)$ a monoid and $\mu\colon{M\times A\times M\to A}$ a map satisfying conditions $(\ref{eq: mu41})$--$(\ref{eq: mu44})$, then, as suggested by Proposition \ref{prop: one object}, every $M$-semibimodule induces a 2-cell structure over the monoid  $(M,\cdot,1)$ considered as a one object category. Moreover, every 2-cell structure $(H,\dom,\cod,0,+)$ over $M$ gives rise to a $M$-semibimodule with 
\begin{equation}
    \xymatrix{A=\{a\in H\mid \dom(a)=\cod(a)=1\}\ar@{^{(}->}[r]^(.8){k} & H}
\end{equation}
as soon as the map $0\colon{M\to H}$ splits into $0_{\bullet}$ and $0^{\bullet}$, i.e., $0(f)=0_{\bullet}(f)+0^{\bullet}(f)$, as illustrated,
\begin{equation}
    \xymatrix{
{\bullet} 
\ar@<1ex>@/^2pc/[rr]^{f}="1"
\ar@<-1ex>@/_2pc/[rr]_{f}="3"\ar[rr]|{1}="2"    && {\bullet}  \\
\ar@{=>}^(.56){0^{\bullet}(f) } "1";"2"
\ar@{=>}^(.45){0_{\bullet}(f) } "2";"3"} =
\xymatrix{
{\bullet} 
\ar@<1ex>@/^2pc/[rr]^{f}="1"
\ar@<-1ex>@/_2pc/[rr]_{f}="3"    && {\bullet} & \\
\ar@{=>}^{0(f) } "1";"3"}
\end{equation}
and there exists a map $q\colon{H\to A}$ such that
\begin{align}
 qk(a)&=a\\
    0_{\bullet}(\cod(x))+kq(x)+0^{\bullet}(\dom(x))&=x\\
     kq(0_{\bullet}(g))&=0(1)=    kq(0^{\bullet}(f))\\
    kq(x+y)&=kq(x)+kq(y)
\end{align}
 for all $a\in A$, $x,y\in H$ with $\dom(x)=\cod(y)$ and $f,g\in M$; in addition
\begin{equation}\label{eq: 0=0+0}
    0(1)=0^{\bullet}(f)+0_{\bullet}(f)
\end{equation}
for all $f\in M$ and equation $(\ref{eq: 43})$ must be satisfied for $\mu=q\mu_H$, as detailed below.  

Indeed, the assumptions above are sufficient to produce a bijection 
\begin{equation}
    \xymatrix{H \ar@{->}@<0.5ex>[r]^(.35){\alpha} & M\times A\times M \ar@{->}@<0.5ex>[l]^(.65){\beta} }
\end{equation}
with $\alpha(x)=(\cod(x),q(x),\dom(x))$ and $\beta(g,a,f)=0_{\bullet}(g)+k(a)+0^{\bullet}(f)$ for all $a\in A$, $x\in H$, $f,g\in M$ thus permitting to define the map $\mu\colon{M\times A\times M\to A}$ by the formula $\mu(u,a,v)=q\mu_H(u,k(a),v)$ with $\mu_H\colon{M\times H\times M\to H}$, as illustrated,
\begin{equation}
\xymatrix{
{\bullet} 
\ar@<1ex>@/^1pc/[rr]^{1}="1"
\ar@<-1ex>@/_1pc/[rr]_{1}="2"    && {\bullet} \ar[dd]^{u} & \\
\\
{\bullet} \ar[uu]^{v}
\ar@<1ex>@/^1pc/[rr]^{uv}="3"
\ar@<-1ex>@/_1pc/[rr]_{uv}="4"    && {\bullet} & \\
\ar@{=>}^{k(a) } "1";"2"
\ar@{=>}|{\mu_H(u,k(a),v) } "3";"4"}
\end{equation}
derived from equation $(\ref{eq: mu})$ as $\mu_H(u,x,v)=H(v,u)(x)$.
Finally, it is readily observed that condition $(\ref{eq: 0=0+0})$ can be discarded, thus giving rise to a slight generalization of a $M$-semibimodule as a system $(A,0,\rho,\mu)$ with 
 \begin{align}
        \mu(u,0,v)&=0\label{eq: mu41a}\\
        \rho(\mu(u,a_1,v),g,\mu(u,a_2,v))&=\mu(u,\rho(a_1,g,a_2),v)\\
        \mu(u',\mu(u,a,v),v')&=\mu(u'u,a,vv')\\
        \mu(1,a,1)&=a\label{eq: mu44a}\\
        \rho(a,f,0)&=a=\rho(0,f,a)\\
        \rho(a_1,f,\rho(a_2,g,a_3))&=\rho(\rho(a_1,f,a_2),g,a_3)
    \end{align}
   where $\rho\colon{A\times M\times A\to A}$ is obtained as
\begin{align}
    k\rho(a_1,g,a_2)=k(a_1)+0^{\bullet}(g)+0_{\bullet}(g)+k(a_2).
\end{align}

The primary concept driving this project is the analysis of a category $\C$ through the examination of its behavior under various 2-cell structures. The subsequent result indicates a general procedure for constructing arbitrary 2-cell structures over a base category.

\section{Parameterized 2-cell structures on arbitrary categories}

In this section, a general procedure for constructing parameterized 2-cell structures within arbitrary categories is presented.
Based on the previous result, and to parameterize 2-cell structures over a base category $\C$, it becomes necessary to introduce a novel and abstract categorical structure on each object $B$ in $\C$. This abstraction was derived after comprehensive analysis aimed at encompassing a broad range of examples arising from various specific situations, including, but not limited to, the following:
\begin{enumerate}
    \item $(B,+,0)$ a monoid;
    \item $(B,+,0,\overline{()})$ a group;
    \item $(B,m)$ a Mal'tsev algebra with $m\colon B^3\to B$;
    \item $(B,+,\overline{()})$ an inverse semigroup;
    \item $(B,R,m,\overline{() })$ a Brandt groupoid with $R\subseteq B\times B$ and $m\colon{R\to B}$.
\end{enumerate}

The result is a category that will be denoted $\Bla(\C)$ whose objects are four-tuples $(B,R,m,e)$ with $B$ an object in $\C$, $R\hookrightarrow B\times B \times B$ a ternary relation on $B$ and $e\colon{B\to B}$, $m\colon{R\to B}$ morphisms in $\C$. Note that  if there is no canonical choice for products in $\C$, or even in the event that products may not always exist, it is nevertheless possible to define $R$ by specifying three projection morphisms $\pi_1,\pi_2,\pi_3\colon{R\to B}$ as part of the structure with the property that the triple $(\pi_1,\pi_2,\pi_3)$ is jointly monic. Moreover, each system $(B,R,m,e)$ has to satisfy the following two conditions:
\begin{enumerate}
    \item for every object $A$ in $\C$ and every three parallel morphisms $x,y,z\colon{A\to B}$, if $\langle x,y,ey\rangle$ and $\langle ey,y,z\rangle$ factor through $R$ then $\langle x,y,z\rangle$ factors through $R$ as well, i.e.,
    \begin{equation} 
    \begin{array}{rl}
& (x,y,ey)\in R \\
\therefore & (ey,y,z)\in R \\
\hline
 & (x,y,z)\in R
\end{array}
\end{equation}
\item for every $x\colon{A\to B}$, the morphism $\langle ex,x,ex\rangle$ factors through $R$, i.e.,
\begin{equation}
    (ex,x,ex)\in R.
\end{equation}
\end{enumerate}

The  2-cell structures parameterized over $\C$ will thus be indexed by sections to the forgetful functor from $\Bla(\C)$ to $\C$. Further details are provided in the following result.

\begin{proposition}\label{prop: split-eip,mono}
    Let $\C$ be a category. Every $\C^{*}$ obtained as a (split-epi,mono) factorization in $\Cat$
    \begin{equation}
        \xymatrix{\Bla(\C)\ar[rr]^{U} \ar@<-.81ex>[rd]_{Q}&&\C\\ & \C^{*}\ar@<-.81ex>[ul]_{R}\ar@{^{(}->}[ru]_{V}},\quad U=VQ, QR=1,  (UR=V)
    \end{equation}
    with $U(B,R,m,e)=B$ the forgetful functor from $\Bla(\C)$ to $\C$, can be described as follows:
    \begin{enumerate}
        \item the objects are those objects in $\C$ that can be equipped with a (chosen) structure $R(B)=(B,R_B,m_B,e_B)$ in $\Bla(\C)$;
        \item the morphisms are morphisms $u\colon{B\to B'}$ in $\C$ for which:
        \begin{enumerate}
            \item $e_{B'}u=ue_B$
            \item if $\langle x,y,z \rangle$ factors through $R_B$ then $\langle ux,uy,uz\rangle$ factors through $R_{B'}$ and $$um_B\langle x,y,z\rangle=m_{B'}\langle ux,uy,uz\rangle.$$
        \end{enumerate}
    \end{enumerate}
    Furthermore, each $R(\C^{*})$ induces a bifunctor
\begin{equation}
    H_{R}\colon{\C^{\text{op}}\times \C^{*}\to \Set} 
\end{equation}
determined as follows:
\begin{enumerate}
    \item for every pair of objects $A$, $B$ in $\C$, the set $H_R(A,B)$ consists of column vectors
    \begin{equation}\nonumber
    \begin{bmatrix}
        f\\
        t\\
        g
    \end{bmatrix}
\end{equation}
with $g,t,f\colon{A\to B}$ morphisms in $\C$ such that $\langle t,f,e_B f\rangle$  and $\langle e_B g,g,t\rangle$ factor through $R_B$. The map $H_{R}(v,u)\colon{H_{R}(A,B)\to H_{R}(A',B')}$ is obtained as
\begin{equation}
H_R(v,u)\left(
    \begin{bmatrix}
        f\\
        t\\
        g
    \end{bmatrix}
\right)=
    \begin{bmatrix}
        ufv\\
        utv\\
        ugv
    \end{bmatrix}
\end{equation}
    which is well defined for all $A'$, $A$, $B$, $B'$ objects in $\C$ as well as morphisms  $v\colon{A'\to A}$ and $f,t,g\colon{A\to B}$ in $\C$, while $u\colon{B\to B'}$ is a morphism in $\C^{*}$.
\end{enumerate}

\end{proposition}
\begin{proof}
    The proof involves a long but straightforward verification thus being omitted. It is clear that the objects $B$ and $B'$ in $\C$ are assumed to be objects in $\C^{*}$ as well.
\end{proof}

For every category $\C$ and every object $B$ in it (for which $B^3$ exists), there are always two possibilities for $R(B)$, namely $(B,B,1_B,1_B)$ and $(B,B^3,\pi_2,1_B)$, which are going to parameterize, not surprisingly, the discrete and co-discrete 2-cells structures over $\C$. However, in general, the best that we can do is to parameterize 2-cell structures over  $\C^{*}$ rather than over $\C$.

We are now in position to analyse, for every $R(\C^{*})$, when does a family of subsets
\begin{equation}
    H(A,B)\subseteq H_R(A,B)
\end{equation}
for which the restriction 
 \begin{equation}\label{eq: mu bis}
        \mu_{A',A,B,B'}\colon{\hom(B,B')\times H(A,B) \times \hom(A',A)\to H(A',B')}
    \end{equation}
is well defined as
\begin{equation}
\mu_{A',A,B,B'}\left(u,
    \begin{bmatrix}
        f\\
        t\\
        g
    \end{bmatrix},v
\right)=
    \begin{bmatrix}
        ufv\\
        utv\\
        ugv
    \end{bmatrix}
\end{equation}
 gives rise to a 2-cell structure over $\C^{*}$ if requiring
 \begin{align}
    \dom_{A,B}  \left( \begin{bmatrix}
        f\\
        t\\
        g
    \end{bmatrix}\right)=& f\\
    \cod_{A,B} \left( \begin{bmatrix}
        f\\
        t\\
        g
    \end{bmatrix}\right)=& g\\
    0_{A,B}(f)=&\begin{bmatrix}
        f\\
        e_B f\\
        f
    \end{bmatrix}\\
    \begin{bmatrix}
        g\\
        t'\\
        h
    \end{bmatrix}+\begin{bmatrix}
        f\\
        t\\
        g
    \end{bmatrix}=&\begin{bmatrix}
        f\\
        m_B\langle t',g,t\rangle\\
        f
    \end{bmatrix}
\end{align}
for every pair of objects $A$, $B$ in $\C$.

\begin{theorem}\label{thm: arbitrary 2-cell structs}
        Let $\C$ be an arbitrary category with $\Bla(\C)$ as described before. For each  (split-epi,mono) factorization in $\Cat$
    \begin{equation}
        \xymatrix{\Bla(\C)\ar[rr]^{U} \ar@<-.81ex>[rd]_{Q}&&\C\\ & \C^{*}\ar@<-.81ex>[ul]_{R}\ar@{^{(}->}[ru]_{V}},\quad U=VQ, QR=1,  (UR=V)
    \end{equation}
    inducing a bifunctor $H_R$ as in Proposition \ref{prop: split-eip,mono}, a family $(H,\dom,\cod,0,+)_{A,B}$ with 
\begin{align}
H(A,B)\subseteq H_R(A,B)\\
    \dom_{A,B}  \left( \begin{bmatrix}
        f\\
        t\\
        g
    \end{bmatrix}\right)=& f\\
    \cod_{A,B} \left( \begin{bmatrix}
        f\\
        t\\
        g
    \end{bmatrix}\right)=& g\\
    0_{A,B}(f)=&\begin{bmatrix}
        f\\
        e_B f\\
        f
    \end{bmatrix}\\
    \begin{bmatrix}
        g\\
        t'\\
        h
    \end{bmatrix}+\begin{bmatrix}
        f\\
        t\\
        g
    \end{bmatrix}=&\begin{bmatrix}
        f\\
        m_B\langle t',g,t\rangle\\
        f
    \end{bmatrix}
\end{align}
    together with a family 
    \begin{equation}
        \mu_{A',A,B,B'}\left(u,
    \begin{bmatrix}
        f\\
        t\\
        g
    \end{bmatrix},v
\right)=H_R(v,u)\left(
    \begin{bmatrix}
        f\\
        t\\
        g
    \end{bmatrix}
\right)
    \end{equation}
    gives a 2-cell structure over $\C^{*}$ if and only if the following conditions hold:
    \begin{enumerate}
        \item if \[\begin{bmatrix}
        f\\
        t\\
        g
    \end{bmatrix}\in H(A,B)\]
    then 
    \[\begin{bmatrix}
        ufv\\
        utv\\
        ugv
    \end{bmatrix}\in H(A',B')\] for all $u$ and $v$ morphisms in $\C^{*}$;
    \item if \[\begin{bmatrix}
        f\\
        t\\
        g
    \end{bmatrix}\in H(A,B)\]
     then $f$ and $g$ are morphisms in $\C^{*}$; 
    \item for every morphism $f\colon{A\to B}$ in $\C^{*}$,
    \[\begin{bmatrix}
        f\\
        e_B f\\
        f
    \end{bmatrix}\in H(A,B)\]
    \item if \[\begin{bmatrix}
        g\\
        t'\\
        h
    \end{bmatrix},\begin{bmatrix}
        f\\
        t\\
        g
    \end{bmatrix}\in H(A,B)\]
    then 
    \[\begin{bmatrix}
        g\\
        t'\\
        h
    \end{bmatrix}+\begin{bmatrix}
        f\\
        t\\
        g
    \end{bmatrix}\in H(A,B)\]
\item  if \[\begin{bmatrix}
        f\\
        t\\
        g
    \end{bmatrix}\in H(A,B)\]
    then $m_B\langle t,f,e_B f\rangle=t =m_B\langle e_B g,g,t\rangle$;
    \item  if \[\begin{bmatrix}
        h\\
        t''\\
        k
    \end{bmatrix},\begin{bmatrix}
        g\\
        t'\\
        h
    \end{bmatrix},\begin{bmatrix}
        f\\
        t\\
        g
    \end{bmatrix}\in H(A,B)\]
    then  $m_B\langle t'',h, m_B\langle t',g,t\rangle\rangle = m_B\langle m_B\langle t'',h,t'\rangle, g,t\rangle$.
    \end{enumerate}
\end{theorem}

\section{Examples of application}\label{sec: examples}

As an example of application let us explore 2-cell structures over $\Mon$. Up to the present, and to the author's best knowledge, the category of monoids has always been considered as a 2-category with 2-cells obtained via the canonical embedding of $\Mon$ into $\Cat$. This example is obtained when a monoid is seen as a one object category and it may suggest that the category of 2-cell structures over $\Mon$ is not a very interesting situation. Indeed, having only three structures and two of them being trivial, namely the discrete and co-discrete, would not be very interesting. However, that is far from being the case as the following application of Theorem \ref{thm: arbitrary 2-cell structs} shows.

In applying Theorem \ref{thm: arbitrary 2-cell structs} to monoids it is convenient to consider $\C^{*}=\Mon$ with $R\colon{\Mon\to \Bla}$ obtained by associating to each monoid $(B,+,0)$ in $\Mon$ a four-tuple $(B,R,m,e)$ in $\Bla$ with $R=B^3$, $m(b_1,b_2,b_3)=b_1+b_3$ and $e(b)=0$. Moreover, it is also convenient to consider as $\C$  the category whose objects are monoids but whose morphisms are maps between the underlying sets. This is convenient since we need to have $\C^{*}\hookrightarrow \C$. The quotient functor $Q$ is not required in the construction and will be left out of our analysis since it would involve a restriction to a suitable subcategory of $\Bla$ which  does not has any effect in our construction.

According to Theorem \ref{thm: arbitrary 2-cell structs}, every family of subsets $H(A,B)$, indexed by pairs of monoids $(A,B)$, such that
\begin{align}
    \left\lbrace \begin{bmatrix}
        f\\
        0\\
        f
    \end{bmatrix}\middle\vert\  f\in\hom_{\C^{*}}(A,B) \right\rbrace
\subseteq H(A,B)\\
H(A,B) \subseteq \left\lbrace \begin{bmatrix}
        f\\
        t\\
        g
    \end{bmatrix}\middle\vert\  f,g\in\hom_{\C^{*}}(A,B), t\in\hom_{\C}(A,B) \right\rbrace
\end{align}
and being closed under composition in $\C^{*}$, that is, if \[\begin{bmatrix}
        f\\
        t\\
        g
    \end{bmatrix}\in H(A,B)\]
    then 
    \[\begin{bmatrix}
        ufv\\
        utv\\
        ugv
    \end{bmatrix}\in H(A',B')\] for all $u$ and $v$ morphisms in $\C^{*}$, induces a 2-cell structure over the base category $\C^{*}=\Mon$.

So, here is a small list of examples:
\begin{enumerate}
\item The set $H(A,B)$ is equal to:
\begin{equation}
  \left\lbrace \begin{bmatrix}
        f\\
        t\\
        g
    \end{bmatrix} \middle\vert\   f,g\in\hom_{\C^{*}}(A,B), t\in\hom_{\C}(A,B) \right\rbrace
    \end{equation}
    \item Same as (1) except that $t$ is required to be a constant map;
\item Same as (1), restricted to $g=t+f$;
\item Same as (1), restricted to $t=0$;
\item Same as (1),  restricted to $t=0$ and $g=f$;
\item Same as (1),  restricted to $g+t=t+f$;
\item Same as (1),  restricted to $g+t=t+f$ and $t$ a constant map;
\end{enumerate}

In each case, $H(A,B)$ may be further restricted to a subset of invertible 2-cells or to a subset of natural 2-cells, or even both; further details can be found in \cite{NMF.15}. 

Case (3) was considered before; case (4) is  co-discrete;  case (5) is discrete; case (7) becomes the canonical case if instead of a constant map $t\colon{A\to B}$ we take it as an element $t_0\in B$ such that $g(x)+t_0=t_0+f(x)$ for all $x\in A$.

If restricting our analysis to groups instead of monoids we observe, in the first place, that most formulas can be re-written in terms of commutators and conjugations such as for example $g=t+f-t$. Moreover, there are two other possibilities in defining $R(B)$ other than the one considered for monoids. Indeed, if $(B,+,0,\overline{()})$ is a group then we can put $e(b)=b$ and $m(b_1,b_2,b_3)=b_1+\overline{b_2}+b_3$ or else we can put $e(b)=\overline{b}$ and $m(b_1,b_2,b_3)=b_1+b_2+b_3$. The same observation is valid for inverse semigroups while for the case of Mal'tsev algebras the only possibility is to put $e(b)=b$ and let $m(b_1,b_2,b_3)$ be the Mal'tsev operation. The case of Brandt groupoids is more involved and will be considered in a forthcoming paper.



{N. Martins-Ferreira \\
Polytechnic of Leiria\\
P-2411-901, Leiria, Portugal\\
              {martins.ferreira@ipleiria.pt}           
}

\end{document}